\documentclass[11pt]{amsart}
\usepackage[T1]{fontenc}
\usepackage[english]{babel}
\usepackage{kpfonts}

\usepackage{amsmath,amssymb,amsthm, enumitem,xspace}
\usepackage{comment}
\usepackage{accents}
\usepackage[colorlinks=true,citecolor=blue,linkcolor=blue]{hyperref}
\usepackage{cleveref}
\usepackage{mathrsfs}

\numberwithin{equation}{section}

\newtheorem{thm}{Theorem}[section]

\newtheorem{lem}[thm]{Lemma}

\newtheorem{ithm}{Theorem}

\newtheorem{icor}[ithm]{Corollary}

\newtheorem*{iprob*}{Problem}

\theoremstyle{definition}

\newtheorem{defi}[thm]{Definition}
\newtheorem{rem}[thm]{Remark}

\newcommand*{\cfrak}{\mathfrak c}

\newcommand*{\sO}{\mathscr{O}}

\newcommand*{\RR}{\mathbf{R}}
\newcommand*{\NN}{\mathbf{N}}

\newcommand*{\CC}{\mathbf{C}}
\newcommand*{\QQ}{\mathbf{Q}}

\newcommand*{\SL}{\mathbf{SL}}

\newcommand*{\GL}{\mathbf{GL}}

\newcommand*{\se}{\subseteq}

\newcommand*{\fhi}{\varphi}

\newcommand*{\Id}{\mathrm{Id}}
\newcommand*{\wt}{\widetilde}

\DeclareMathOperator{\Aut}{Aut}
\DeclareMathOperator{\Sym}{Sym}
\DeclareMathOperator{\Ad}{Ad}

\DeclareMathOperator{\res}{res}
\newcommand*{\pui}[1]{\{\{#1\}\}}
\newcommand*{\QRC}{{\QQ^\mathrm{rc}}}

\title[Lie groups are countably representable]{Lie groups are countably representable}

\author[N. Monod]{Nicolas Monod}
\address{N.M., \'Ecole Polytechnique F\'ed\'erale de Lausanne (EPFL), Switzerland}
\email{nicolas.monod@epfl.ch}

\author[A. Thom]{Andreas Thom}
\address{A.T., Institute of Geometry, TU Dresden, 01062 Dresden, Germany}
\email{andreas.thom@tu-dresden.de}

\begin{document}

\begin{abstract}
We prove that every Lie group can be realized by permutations of a countable set, answering a classical question of Ulam. 

The key deformation construction originated in an argument produced by OpenAI's GPT 6 Astra during extended mathematical discussions.
\end{abstract}

\makeatletter
\@namedef{subjclassname@2020}{\textup{2020} Mathematics Subject Classification}
\makeatother

\maketitle


\section{Introduction}
A group $G$ is \textbf{countably representable} if it can be realized by symmetries of some countable structure; equivalently, if $G$ embeds as a subgroup of the permutation group $\Sym(\NN)$. Besides the trivial case where $G$ itself is countable, there are many natural examples of locally compact or Polish groups in this class, including all non-Archimedean Polish groups and hence, by van Dantzig’s theorem, all totally disconnected second countable locally compact groups.
Less naturally, Schreier and Ulam discovered in 1935 that $\RR$ is countably representable; of course, such a representation has to be discontinuous and non-constructive. This prompted Ulam to ask whether all Lie groups would have the same property~\cite[II.7]{Ulam_book60}, \cite[V.2]{Ulam_book64}. The purpose of this article is to establish a positive answer.

\begin{ithm}\label{thm:Ulam}
Every Lie group is countably representable.
\end{ithm}

The usual convention is understood: Lie groups are second countable, or equivalently have at most countably many connected components (otherwise the statement fails already in dimension zero).

As proved in Proposition~2 of~\cite{Monod_Ulam}, Theorem~\ref{thm:Ulam} implies the following more general statement.

\begin{icor}\label{cor:LC}
Every locally compact second countable topological group is countably representable.\qed
\end{icor}

The only obvious necessary condition for countable representability is to have cardinality at most the continuum. Any Polish group has cardinality continuum (or countable in the discrete case) and would thus be a candidate. However, it is known that some Polish groups fail to be countably representable, for instance the group of homeomorphisms of the line (using~\cite{Rosendal-Solecki}). On the other hand, Corollary~\ref{cor:LC} states that all Polish groups that are locally compact are countably representable.

\subsection{Strategy of the proof}
The state of the art is as follows. On the one hand, all \emph{linear} Lie groups are countably representable by a result of Thomas \cite[\S2]{Thomas99} also discovered in~\cite{Kallman00} and ~\cite{Ershov-Churkin}. On the other hand, \cite{Conversano-Monod} reduces the general question to the case of connected \emph{simple} Lie groups. The remaining difficulty is that a connected simple Lie group is in general not linear, but only a central extension of a linear group.

This difficulty is substantial for at least two reasons. The first is that central extensions, even by a center of order two, can destroy countable representability; see the discussion in~\cite[\S5]{Monod_Ulam}. The second is that the obstruction is reflected in the cohomology of subgroups of countable index, which are by necessity ``wild'' and not concretely definable, as follows from~\cite{Solovay}.

\medskip

In order to explain how we overcome this central extension difficulty, it is useful to recall the only known proof method in the linear case. To produce subgroups of countable index in matrix groups, the strategy is to restrict the coefficients to (non-unital) subrings of $\RR$ that are of countable index (as additive subgroups). View $\RR$ in $\CC$ and apply Steinitz's theorem~\cite{Steinitz} to identify $\CC$ with \emph{any} algebraically closed field of characteristic zero and cardinality of the continuum. Such fields include examples with a natural countable valuation, for instance Puiseux series over $\overline{\QQ}$. The corresponding valuation ring and its ideals will provide the desired subrings of countable index.

This approach is highly non-constructive and will a priori be completely incompatible with the cohomological obstruction that arises for central extensions. Our proof will therefore rely on a much more detailed approach to subrings of $\RR$. Indeed we obtain a continuous family of subrings, in a suitable sense, interpolating between a ring of countable index and the zero ring. This continuous contraction will imply that the cohomological obstruction must vanish on the corresponding subgroups. The precise statement, which may be of independent interest, is as follows.

\begin{ithm}\label{thm:path}
Let $F \se \RR$ be a countable real-closed subfield. There exists a unital $F$-subalgebra $A\se \RR$ of countable index (as subgroup) and a one-parameter family of unital ring homomorphisms
\begin{equation*}
\fhi_\tau\colon A\longrightarrow\RR \kern3mm (0\leq \tau \leq1)
\end{equation*}
each fixing $F$ pointwise and such that:
\begin{enumerate}[label=\normalfont(\roman*)]
 \item $\fhi_0(A)=F$;
 \item $\fhi_1$ is the inclusion $A\hookrightarrow\RR$;
 \item for every $x\in A$, the map $\tau\mapsto\fhi_\tau(x)$ is continuous on $[0,1]$.
\end{enumerate}
\end{ithm}

\noindent
The non-unital subring of $\RR$ that will be used to construct subgroups is then $I=\ker \fhi_0$ and the announced continuous family of non-unital subrings interpolating from~$0$ to $I$ is $\fhi_\tau(I)$. In our applications, it is sufficient to take $F$ to be the real closure $\QRC = \overline\QQ \cap \RR$ of the rationals.

\medskip

One of the fundamental inputs to prove Theorem~\ref{thm:path} is to consider Puiseux series not over $\overline{\QQ}$ but over a field that already consists of functions of a variable~$t$. This variable will be used for the parameter $\tau$ of the statement, although the actual proof will be more complicated and involve gluing together infinitely many such constructions. Let's be more specific:

The first step is to consider $\RR$ as the real closure of a purely transcendental extension $F \subset F(T).$ Naively, for an element of the basis $X \in T$, viewed as a real number, we would pick a rational approximation $a_n \to X$. The aim is to define a map $j_n \colon F[T] \to C([0,1],\RR)$ such that $j_n(X)$ is just the convex interpolation between $a_n$ and $a_{n+1}$. The extension to an element in the real closure $x \in \RR$ can be defined, if the approximations of all the finitely many relevant elements in $T$ lie in a suitable neighborhood prescribed by the implicit function theorem applied to the minimal polynomial of $x$. This naive approach cannot ensure algebraic independence of the values assigned to the elements of $T$ and hence does not extend to $F(T)$ or its real closure; instead we use an infinitesimal perturbation and assign to $X \in T$ an element in $k\pui{s}$, where $k$ is the algebraic closure of the field of fractions of a suitable ring of analytic functions defined on a neighborhood of $[0,1]$. This is Lemma \ref{lem:extension}, which is the technical heart of the proof. It uses a blend of methods that we have not seen in this combination before. Part of it is the local construction of algebraic-analytic deformations preserving polynomial relations, which already appears in Parusiński--Rond~\cite[Theorem~4]{Parusinski-Rond}.
In particular, it crucially uses Lemma \ref{lem:generic}, that allows to upgrade identities of algebraic real-analytic functions that hold at a generic point to actual identities;
a result that should also be attributed to Parusiński–Rond \cite[proof of Theorem 4, p. 344]{Parusinski-Rond}. These ideas are combined with the use of a new infinitesimal deformation parameter coming from the field of Puiseux series. The field $k\pui{s}$ is still small enough to admit countable index subrings that map to $C([0,1],\RR)$ extending the assignment above.
We then glue this sequence of maps along their endpoints, whose values remain in $F$ for a subring of countable index. Each compatibility condition defines a further subring of countable index, and every real number satisfies all but finitely many of these conditions. Consequently, imposing all conditions simultaneously still leaves a subring of countable index.

The second step is more standard and uses the subring of countable index together with the constructed homotopy in order to trivialize the restriction of the central extension to the subring, or more generally, trivialize the cohomology classes we study; this will be made explicit in Theorem~\ref{thm:coho} below.

\subsection*{AI disclosure:} The key countable-index deformation was first constructed by OpenAI's GPT 6 Astra during extended mathematical discussions. The argument was subsequently developed, checked and substantially rewritten by the authors; see the acknowledgments for its history.

\section{Proof of Theorem \texorpdfstring{\ref{thm:path}}{\getrefnumber{thm:path}}}
\subsection{Puiseux series over fields of analytic functions}\label{sec:Puiseux}
First we recall the classical construction of Puiseux series, which is also used in Thomas's proof. We start with a field $k$ that is assumed countable, algebraically closed and of characteristic zero. Later on, we will choose $k$ as a suitable field constructed from analytic paths that are algebraic over rational functions in the real-closed field $F$ of Theorem~\ref{thm:path}. The main goal is to give better embeddings of $\RR$ into the Puiseux field of $k$ than one coming from a direct application of Steinitz's theorem.

\medskip

The field $k\pui{s}$ of Puiseux series with coefficients in $k$ is the collection of all formal Laurent series over $k$ in some root of the formal parameter $s$. In other words, $k\pui{s} = \bigcup_{m\in\NN} k((s^{1/m}))$. The Newton--Puiseux theorem states that $k\pui{s}$ is algebraically closed because $k$ is so and because we are in characteristic zero.

The valuation $v\colon k\pui{s} \to \QQ \cup \{+\infty\}$ is the exponent of the smallest non-zero term, with $v(0)=+\infty$. Thus, the valuation ring $\sO_{k\pui{s}}$ defined by $v\geq 0$ is of countable index in $k\pui{s}$ as an additive subgroup since the coset of a Puiseux series modulo $\sO_{k\pui{s}}$ is determined by the finitely many coefficients in $k$ that have a negative exponent. The residue map $\res\colon \sO_{k\pui{s}} \to k$ associates to a series its constant coefficient; it is a $k$-algebra morphism.

\medskip
We now start with a countable real-closed field $F\se \RR$ and proceed to define $k$. Consider the field of rational functions over $F$, denoted $F(t)$ to distinguish the formal variable $t$ from the variable $s$ introduced above. We denote by $C^\omega_{F(t)}([0,1])$ the ring of real-valued functions on $[0,1]$ that extend real analytically to an open interval containing $[0,1]$ and are algebraic over $F(t)$. The latter condition is well-defined because we use $t$ as the variable of our analytic functions. We observe that the evaluation morphism at $t=0$ or $t=1$ maps $C^\omega_{F(t)}([0,1])$ to $F$. Indeed, since the elements of $C^\omega_{F(t)}([0,1])$ are algebraic over $F(t)$, the evaluated values are algebraic over $F$ and real, hence belong to $F$.

Finally, $k$ is defined to be an algebraic closure of the field of fractions of $C^\omega_{F(t)}([0,1])$, which we can form since analytic functions form an integral domain. This field is still countable, as required above.

\medskip
Here and below, we denote the continuum cardinal by $\cfrak$. We can choose a family $(u_\alpha)_{\alpha<\cfrak}$ in $k[[s]]^\times$ that is algebraically independent over $k(s)$. Indeed, if $(v_\alpha)_{\alpha<\cfrak}$ is a transcendence
basis of $k((s))$ over $k(s)$, then we set $u_\alpha=s^{-v_s(v_\alpha)}v_\alpha$, which remains algebraically independent over $k(s)$.

\subsection{Algebraic approximations along arcs}
We choose once and for all a transcendence basis $T=(X_\alpha)_{\alpha<\cfrak}$ of $\RR$ over $F$ and an assignment $x\mapsto T_x$ which to each $x\in \RR$ associates a tuple $T_x$ of elements of $T$ such that $x$ is algebraic over $F(T_x)$. In particular, we can view $T_x$ as a vector in $\RR^q$ for some $q$ (which is chosen depending on $x$). We moreover choose for each $x\in \RR$ a convex open neighbourhood $W_x\se \RR^q$ of $T_x$ and a real-analytic function $g_x$ on $W_x$ satisfying $g_x(T_x) = x$, as follows. Considering the minimal polynomial of $x$ over $F(T_x)$, the analytic implicit-function theorem (see Theorem~2.3.5 in~\cite{Krantz-Parks}) provides a real-analytic algebraic branch $g_x$ for the root $x$ in a suitable neighbourhood $W_x$ of $T_x$. We choose $W_x$ small enough to be convex and to ensure that the root $x$ remains simple there; this is possible because the coordinates of $T_x$ are algebraically independent by definition. 

We will use below the following principle, which can be found in work of Parusiński–Rond \cite[proof of Theorem 4, p. 344]{Parusinski-Rond}. We include the proof adapted to our setting for convenience:

\begin{lem}\label{lem:generic}
Let $U \se \RR^q$ be a domain and let $f$ be a real-analytic function on $U$ that is algebraic over $F(z_1, \ldots, z_q)$. If $f$ vanishes at a point whose coordinates are algebraically independent over $F$, then $f$ vanishes identically on $U$.
\end{lem}

\begin{proof}
Let $\xi\in U$ be the vanishing point in the statement and write $z=(z_1, \ldots, z_q)$. By the real-analytic identity theorem, it suffices to prove that the germ of $f$ at $\xi$ is zero. If not, then its minimal polynomial over $F(z)$
\begin{equation*}
f^d+r_{d-1}(z)f^{d-1}+\cdots+r_0(z)=0
\end{equation*}
has nonzero constant coefficient $r_0$. Algebraic independence implies that every nonzero polynomial over $F$ is nonzero at $\xi$. Thus all the rational functions $r_j$ are defined there and $r_0(\xi)\ne0$. Evaluating the equation contradicts $f(\xi)=0$.
\end{proof}

Next, we choose rational approximations for each $X_\alpha\in T$, that is, sequences $(a_{\alpha, n})_{n\in \NN}$ in $\QQ$ with $a_{\alpha, n} \to X_\alpha$. We then define the affine functions $\ell_{\alpha,n}$, considered as elements of $C^\omega_{F(t)}([0,1])$ and hence also of $k$, by
\begin{equation*}
\ell_{\alpha,n} (t) = (1-t) a_{\alpha, n}  + t a_{\alpha, n+1}.
\end{equation*}
Finally, we define $F$-morphisms
\begin{equation*}
j_n\colon F(T) \longrightarrow k\pui{s}
\end{equation*}
by the requirement
\begin{equation*}
j_n(X_\alpha) = \ell_{\alpha,n} + s u_\alpha \kern3mm \text{for} \kern3mm  X_\alpha\in T,
\end{equation*}
where $u_\alpha\in k[[s]]^\times$ was chosen at the end of Section~\ref{sec:Puiseux}. That choice ensures in particular that all $\ell_{\alpha,n} + s u_\alpha$ are algebraically independent over $F$ so that $j_n$ extends indeed to $F(T)$.

\begin{defi}
We say that $x\in \RR$ is \textbf{good at $n\in \NN$} if both tuples $a_{T_x,n}$ and $a_{T_x, n+1}$ belong to $W_x$.
\end{defi}
In this definition,  $a_{T_x, n}$ is a shorthand for the tuple $(a_{\alpha_1, n}, \ldots , a_{\alpha_q, n})$ when $T_x=(X_{\alpha_1}, \ldots , X_{\alpha_q})$ and likewise for $n+1$. Similarly, we will write $\ell_{T_x,n}$ below.

\smallskip
Note that every given $x$ is good at all sufficiently large $n$ and that this condition implies that $\ell_{T_x,n}$ remains in $W_x$. Thus we can define
\begin{equation*}
 f_{x,n}(t)=g_x(\ell_{T_x,n}(t)).
\end{equation*}
We claim that $f_{x,n}\in C^\omega_{F(t)}([0,1])$. By construction, $f_{x,n}$ is real analytic on an open interval containing $[0,1]$. Substitution of the rational affine functions $\ell_{\alpha,n}$ into the monic equation for $g_x$ gives a monic equation over $F(t)$ whose coefficient denominators do not vanish along the segment, whence the claim.

\medskip

The very definition of $f_{x,n} $ further implies the following.

\begin{lem}\label{lem:tail}
For every $x\in \RR$ we have:
\begin{enumerate}[label=\normalfont(\roman*)]
\item $f_{x,n}(1)=f_{x,n+1}(0)$ for all sufficiently large $n$,\label{lem:tail:coincide}
\item $\lim\limits_{n\to\infty}\,\sup\limits_{0\leq t\leq1}|f_{x,n}(t)-x| = 0$.\qed \label{lem:tail:limit}
\end{enumerate}
\end{lem}

\subsection{A first embedding}
At this point we will obtain our first embedding of the field $\RR$ into $ k\pui{s}$ by extending the embedding $j_n$ of $F(T)$.

\begin{lem}\label{lem:extension}
For every $n$ there is an embedding $J_n:\RR\hookrightarrow k\pui{s}$
extending $j_n$ such that, for every $x\in\RR$ good at $n$,
\[
 J_n(x)\in\sO_{k\pui{s}},\qquad
 \res(J_n(x))=f_{x,n}.
\]
\end{lem}

\begin{proof}
First, we define $J_n(x)$ for any $x$ that is good at $n$. Since $j_n(X_\alpha)$ is $\ell_{\alpha,n} + s u_\alpha$, the only choice is as follows. Consider the multivariate Taylor series of $g_x$ at the point $\ell_{T_x,n}(t)$. Evaluate that series at the variables $s u_\alpha$, where $\alpha$ indexes the coordinates of $T_x$. This defines a formal series $J_n(x)$ in $s$ whose coefficients depend on $t$. We claim that these coefficients of the $s$-series lie in $k$. Indeed, implicit differentiation expresses each derivative occurring in the definition of Taylor series as a rational expression in the coordinates, the branch value, and the polynomial coefficients together with their derivatives. Thus every Taylor series coefficient for the variables $s u_\alpha$ is a ratio of functions that are real analytic on a neighborhood of $[0,1]$ and algebraic over $F(t)$. Such ratios define elements of the field of fractions of $C^\omega_{F(t)}([0,1])$, which lies in $k$. This shows that, when viewed as a series in $s$, the coefficients are indeed in $k$.

Thus $J_n(x)$ belongs to $k[[s]] \se \sO_{k\pui{s}}$ and its residue is the constant term of the Taylor series, namely, $\res(J_n(x))=f_{x,n}$, which is in $C^\omega_{F(t)}([0,1])$ as already observed after the definition of $f_{x,n}$.

\smallskip
Next, we want this choice of $J_n(x)$ for good $x$ to extend to a morphism
\begin{equation}\label{eq:Jn}
F(T)[\,x\in\RR:x\text{ good at }n\,]\longrightarrow k\pui{s}.
\end{equation}
For this to be possible, we need to check that polynomial relations over $F(T)$ are preserved. 
The compatibility argument used below follows the method of Parusiński--Rond~\cite[Theorem~4]{Parusinski-Rond} that we repeat in adapted form for convenience. Consider thus a finite family $(x_1, \ldots, x_r)$
of good elements and a polynomial relation among them with coefficients in $F(T)$; clear the coefficient denominators. 

Denote by $A$ the (finite) collection of all indices $\alpha$ with $X_\alpha\in T_{x_i}$ for some $i=1, \ldots, r$ together with all indices occurring in the coefficients of the relation. We obtain a convex domain $U$ in $\RR^A$ by intersecting the pull-back of all $W_{x_i}$. Write the cleared relation as
\[
 Q(X_A,x_1,\ldots,x_r)=0,
\]
where $Q$ is a polynomial with coefficients in $F$. The function
\[
Q(z_A,g_{x_1}(z),\ldots,g_{x_r}(z))
\]
on $U$ is analytic,
algebraic over $F(z_A)$, and zero at the algebraically independent
tuple $X_A$. Therefore it is identically zero by Lemma~\ref{lem:generic}.
This forces the local branches to respect every relation between their
actual real values, on the entire common neighborhood containing the arc.

Taking Taylor series along the arc preserves this identity; therefore,
\[
Q\big(j_n(T_A), J_n(x_1),\ldots, J_n(x_r)\big)=0 \kern3mm \text{in} \kern3mm k[[s]].
\]
Thus the original relation, including its base-field coefficients, is preserved. This justifies that $j_n$ extends indeed to a $F(T)$-algebra homomorphism as in~\eqref{eq:Jn} above.

Finally, since $\RR$ is algebraic over $F(T)$, the ring $F(T)[\,x\in\RR:x\text{ good at }n\,]$ is a field, so that this extension~\eqref{eq:Jn} is injective. Using moreover that $k\pui{s}$ is algebraically closed, this embedding extends to $\RR$ as stated.
\end{proof}

\subsection{Gluing the embeddings}
We are now ready to prove Theorem~\ref{thm:path}. We define unital $F$-subalgebras $M_n$ of $\RR$ by
\begin{equation*}
M_n = \{ x\in \RR : J_n(x) \in \sO_{k\pui{s}} \text{ and } \res J_n(x) \in C^\omega_{F(t)}([0,1]) \},
\end{equation*}
where we recall that $C^\omega_{F(t)}([0,1])\se k$ is the ring of analytic functions introduced in the construction of $k$. Since $k\pui{s}/\sO_{k\pui{s}}$ and $k$ are countable, it follows that $M_n$ has countable index in $\RR$ (here and always: as additive subgroup).

Since the elements of $C^\omega_{F(t)}([0,1])$ are functions of a variable $t$, we can use the evaluations $t=0,1$ and further define
\begin{equation*}
A_n = \big\{ x\in M_n \cap M_{n+1} : (\res J_n(x) )(1) =  (\res J_{n+1}(x) )(0) \big\}.
\end{equation*}
Finally, $A = \bigcap_{n\in\NN} A_n$. Note that all $A_n$ and hence also $A$ are unital $F$-subalgebras of $\RR$.

The intersection $M_n \cap M_{n+1}$ has countable index in $\RR$ and hence so does $A_n$ because it is the kernel in $M_n \cap M_{n+1}$ of an additive map to $F$, which is countable. We claim that $A$ also has countable index in $\RR$. Indeed, by Lemma~\ref{lem:tail}\ref{lem:tail:coincide}, every $x\in \RR$ belongs to $A_n$ when $n$ is large enough; therefore, the quotient group $\RR/A$ embeds into $\bigoplus_{n\in\NN} \RR/A_n$, which is countable.

\medskip
At this time, for $x\in A$, we can glue together the maps $t\mapsto  (\res J_n(x) )(t)$ as follows. Define $\Phi_\tau\colon A \to\RR$ for $\tau\in\RR_{\geq 0}$ by
\begin{equation*}
\Phi_\tau (x) =  \big(\res J_n(x) \big)(\tau - n), \kern3mm n=\lfloor\tau \rfloor
\end{equation*}
where $n=\lfloor\tau \rfloor$ is the integer part of $\tau$. The definition of $A_n$ ensures that for a given $x$ this is continuous as a function of $\tau\in\RR_{\geq 0}$ and each $\Phi_\tau$ is a unital ring homomorphism fixing $F$; moreover, $\Phi_0(A)=F$ holds. We will obtain the path $\fhi_\tau$ of Theorem~\ref{thm:path} from any continuous change of variable identifying  $\RR_{\geq 0}$ with $[0, 1)$ provided we justify 
\begin{equation*}
\forall\, x\in A: \ \lim_{\tau\to\infty} \Phi_\tau (x) = x.
\end{equation*}
This convergence follows from Lemma~\ref{lem:tail}\ref{lem:tail:limit}.

\section{Proof of Theorem \texorpdfstring{\ref{thm:Ulam}}{\getrefnumber{thm:Ulam}}}

Let $G$ be a connected semisimple Lie group. Consider the adjoint group $\Ad(G)$ of $G$. Namely, denoting the Lie algebra if $G$ by $\mathfrak g$, this is the image $\Ad(G)<\Aut(\mathfrak g)$ of $G$ in the automorphism group of $\mathfrak g$ under the adjoint representation. As Lie groups, $\Ad(G)\cong G/Z_G$, where $Z_G$ denotes the center of $G$.

The new contribution that is needed to solve Ulam's problem is the following:

\begin{thm}\label{thm:new}
There exists a subgroup of countable index in $\Ad(G)$ that can be lifted to $G$.
\end{thm}

To be perfectly unambiguous, the lifting of a subgroup $\Gamma<\Ad(G)$ means that there exists a subgroup $\wt \Gamma< G$ that maps bijectively onto $\Gamma$ under the adjoint representation.

For the purposes of applying the results of~\cite{Conversano-Monod}, we could further assume that $G$ is simple and that the center $Z_G$ is finite, but this would not simplify the argument below.

\begin{proof}[Proof that Theorem~\ref{thm:new} implies Theorem~\ref{thm:Ulam}]
In view of Theorem~1.5 in~\cite{Conversano-Monod}, it suffices to show that every connected simple Lie group with finite center is countably representable. We shall show that this holds for our given connected semisimple Lie group $G$.

Since $\Ad(G)$ is linear, we know from Thomas's theorem that it acts faithfully on some countable set $X$. Let now $\Gamma<\Ad(G)$ be a subgroup of countable index in $\Ad(G)$ that can be lifted to a subgroup $\wt \Gamma< G$. It suffices to verify that the $G$-action on $G/\wt \Gamma \times X$ is faithful. The kernel of the action on $G/\wt \Gamma$ meets $Z_G$ trivially, whereas the kernel of the $G$-action on $X$ is exactly $Z_G$. This confirms that the action on the product is faithful.
\end{proof}

\begin{rem}
In fact, Theorem~\ref{thm:new} is strong enough that one does not need to appeal to Thomas's result as above. Instead, after reducing to the case where $G$ is simple as a Lie group, one could directly act on $G/\wt \Gamma$ and use the fact that $\Ad(G)$ is simple as an abstract group, as proved by Goto, Corollary~3 in~\cite{Goto69}.
\end{rem}

We now discuss our approach to Theorem~\ref{thm:new}. Choosing a basis for $\mathfrak g$, we can consider $\Ad(G)$ as a matrix group. More precisely, $\Ad(G)$ is the identity component if the linear algebraic group $\Aut(\mathfrak g)$ defined by the structure equations associated to the chosen basis of $\mathfrak g$.

The subgroup $\Gamma<\Ad(G)$ will be obtained by restricting the coefficients in $\Ad(G)$ to a carefully chosen non-unital subring $I$ of $\RR$ in this matrix representation. As usual, this means that we consider matrices $M$ such that $M-\Id$ has all coefficients in $I$, and $\Gamma$ will have countable index iff $I$ has countable index an an additive subgroup of $\RR$. The entire difficulty is to choose $I$ is such a way that $\Gamma$ can be lifted to the central extension $G$ of $\Ad(G)$.

Equivalently, the goal is to find $I$ in such a way that the corresponding lifting obstruction in the degree two ``continuous'' cohomology of the topological group $\Ad(G)$ vanishes when restricted to $\Gamma$. We recall that even though the cohomology of \emph{topological} groups is notoriously delicate with a number of different definitions, the situation is particularly clear when the coefficients are a trivial discrete module as in our case, where the coefficients are $Z_G$. An excellent modern source for the relevant definitions and references is~\cite{Austin-Moore} (despite its name, one should not think of continuous cohomology as being represented by continuous cochains on $G$ in general.)

We shall in fact obtain a stronger result than this degree-two vanishing:

\begin{thm}\label{thm:coho}
There exists a non-unital subring $I$ of $\RR$ of countable index such that every continuous cohomology class of $\Ad(G)$ for any discrete trivial coefficients in any positive degree vanishes when restricted to the subgroup of countable index $\Gamma<\Ad(G)$ corresponding to $I$.
\end{thm}

This indeed contains Theorem~\ref{thm:new} because the obstruction to lift a subgroup is the pull-back of the class in the second continuous cohomology of $\Ad(G)$ with coefficients in the trivial discrete module $Z_G$ by the very definitions of Moore and Segal. See in particular Propositions~3.3 and~4.3 in~\cite{Segal70} (the latter on page 384; on the previous page another statement is also mistakenly labeled as Proposition~4.3 instead of~4.2). See also Theorem~10 in~\cite{Moore76}.

Nevertheless, we comment in Remark~\ref{ref:ext} below how the central extension case can be understood without using the full power of the Moore--Segal--Wigner cohomology.

\medskip
In conclusion, it only remains to establish:

\begin{proof}[Proof of Theorem~\ref{thm:coho}]
Let first $\Gamma$ be any group (without topology) and let $f_\tau \colon\Gamma\to \Ad(G)$ be a family of group homomorphisms indexed by $\tau\in[0, 1]$ such that $f_\tau(\gamma) \in \Ad(G)$ depends continuously on $\tau$ for any given $\gamma\in \Gamma$. Let $Z$ be an abelian group, considered as a discrete trivial module for both $\Ad(G)$ and $\Gamma$. One very concrete way to implement the pull-back through $f_\tau$ from the continuous cohomology of $\Ad(G)$ to the ordinary cohomology of $\Gamma$ is to use Moore's Borel cohomology $H^*_\mathrm{M}(\Ad(G), Z)$, in which case a $Z$-valued Borel cocycle can be restricted to a cocycle on $f_\tau(\Gamma)$, compare e.g. pp.~27--28 in~\cite{Moore76}. In order to use the continuity of the family $f_\tau$, we recall that in the present setting $H^*_\mathrm{M}(\Ad(G), Z)$ is isomorphic to the sheaf cohomology $H^*_\mathrm{Sh}(B_{\Ad(G)}, Z)$ of Milnor's join $B_{\Ad(G)}$ with coefficient in the constant sheaf also denoted by $Z$. This is proved by Wigner in~\cite{Wigner_PhD, Wigner73} but can be read also in Theorem~E of~\cite{Austin-Moore}. In the discrete case, the corresponding fact for $B_\Gamma$ is completely classical. Under these isomorphisms, the pull-back will be induced by the maps $Bf_\tau\colon B_\Gamma \to B_{\Ad(G)}$ determined by $f_\tau$. Our assumption provides a homotopy between $Bf_0$ and $Bf_1$, compare Theorem~8.1 in~\cite{Dold63}, and this implies that the corresponding maps in cohomology are identical, see e.g. Theorem~IV.1.1 in~\cite{Iversen}.

\medskip
We now apply this to the following set-up for Theorem~\ref{thm:coho}. Let $F=\QRC$ be the real-closure of the rationals and consider the $F$-subalgebra $A\se \RR$ provided by Theorem~\ref{thm:path} together with the path $\fhi_\tau\colon A\to\RR$ of unital ring homomorphisms. Define $I=\ker \fhi_0$, which is a non-unital subring of $\RR$ that has countable index as a subgroup. 

In order to realize $\Aut(\mathfrak g)$ and $\Ad(G)$ as matrix groups, we need to choose a basis of the Lie algebra. We can take the basis in such a way that the structure constants are integers, see Proposition~14.3 in~\cite{Raghunathan}. (This is not crucial: otherwise we could enlarge $F$ to contains all structure constants). Writing $d=\dim(\mathfrak g)$, the group $\Aut(\mathfrak g)$ is now an algebraic subgroup of $\GL_d(\RR)$. The group $\Aut(\mathfrak g)$ preserves the Killing form, which is non-degenerate by Cartan's characterization of semi-simplicity~\cite[I.1.42]{Knapp02}; therefore, $\Aut(\mathfrak g)$ lies in fact in the subgroup $\SL^\pm_d(\RR)$ of matrices with determinant $\pm 1$. Thus $\Ad(G)$ is in $\SL_d(\RR)$, and it is in fact the connected component of the identity of $\Aut(\mathfrak g)$, see e.g. Corollary~6.5 in~\cite[II]{Helgason01}.

\medskip
Formally, restricting coefficients to $I$ means the following. Denote by $\SL_d(\RR, I)$ the collection of matrices $M$ satisfying that $M-\Id$ has all its coefficients in $I$. This is a subgroup because $I$ is a subring. The fact that $I$ has countable index in $\RR$ as an additive subgroup implies that  $\SL_d(\RR, I)$ has countable index in $\SL_d(\RR)$, e.g. by considering that root subgroups generate the latter (compare also Remark~1 in~\cite{Ershov-Churkin}). Therefore, the group $\Gamma = \Ad(G) \cap \SL_d(\RR, I)$ has countable  index in $\Ad(G)$.

We now define the homomorphisms $f_\tau \colon\Gamma\to \Ad(G)$ by applying entry-wise $\fhi_\tau$ to the matrix coefficients of $\Gamma$; we need to justify that the resulting matrix still lies in $\Ad(G)$. The structure equations defining $\Aut(\mathfrak g)$ within $\GL_d(\RR)$ belong to $F$, so that  $f_\tau$ ranges in $\Aut(\mathfrak g)$. On the other hand, since $\Ad(G)$ is the identity component of that group and since $f_1(\Gamma) = \Gamma \se \Ad(G)$, continuity of each individual $\tau\mapsto f_\tau(\gamma)$ implies that all of $f_\tau(\Gamma)$ remains in $\Ad(G)$.

At this point, we are in the setting considered in the first paragraph of the proof. The original inclusion of $\Gamma$ in $\Ad(G)$ is realized by $f_1$, and we conclude that the pull-back to $\Gamma$ of the cohomology of $\Ad(G)$ can equivalently be done through $f_0$. However, $f_0(\Gamma)$ is the trivial group; the statement follows.
\end{proof}

\begin{rem}\label{ref:ext}
The proof of the special case of Theorem~\ref{thm:new} can be phrased elementarily as follows. Consider the subgroup $\Gamma <\Ad(G)$ defined above; we need to show that it can be lifted to $G$. The adjoint representation is, topologically, a covering map $G \to \Ad(G)$. Therefore, given any $\gamma\in \Gamma$, the homotopy lifting property and $f_0(\Gamma)=\{e\}$ implies that there is a unique lift of paths
\begin{equation*}
\tau \longmapsto \wt f_\tau(\gamma) \in G \kern3mm \text{such that}\kern3mm \Ad\big(\wt f_\tau(\gamma) \big) = f_\tau(\gamma)  \kern3mm \text{and}\kern3mm \wt f_0(\gamma) =e.
\end{equation*}
By uniqueness, it follows that $\wt f_\tau$ is a group homomorphism for any given $\tau$. Now $\wt f_1(\Gamma)$ is the desired group lift of $\Gamma$.
\end{rem}

\section*{Acknowledgments}
The proof emerged from three weeks of discussions with OpenAI's GPT, initially aimed at exploring $\SL_3(\RR)$ with ring deformations and $K$-theoretic arguments.
GPT Astra eventually produced the countable-index deformation underlying Theorem~\ref{thm:path}. We then suggested extending the argument to $\SL_n(\RR)$ and other simple Lie groups; combined with the reduction results of \cite{Conversano-Monod}, these extensions led to the general theorem.

The proof was substantially rewritten and checked by both authors. During this process, they identified relevant references in which substantial parts of the argument already appear, notably the algebraic-analytic construction of Parusiński--Rond \cite{Parusinski-Rond}. The simultaneous Puiseux embedding with prescribed residues in Lemma~\ref{lem:extension}, and its combination with the countable-index gluing argument, appear to be new and creative.


\bibliographystyle{amsalpha}
\bibliography{ulam}

\end{document}